\documentclass[reqno]{amsart}

\usepackage{a4wide}
\usepackage{color}
\usepackage{mathrsfs}
\usepackage{mathtools}
\usepackage{amsmath}
\usepackage{amssymb}
\numberwithin{equation}{section}
\usepackage[colorlinks,citecolor=green,linkcolor=red]{hyperref}

\usepackage[latin1]{inputenc}

\newcommand{\N}{\mathbb{N}}
\newcommand{\R}{\mathbb{R}}

\newcommand{\sfd}{{\sf d}}
\renewcommand{\d}{{\mathrm d}}
\newcommand{\e}{{\rm e}}
\newcommand{\X}{{\rm X}}
\newcommand{\LIP}{{\rm LIP}}
\newcommand{\lip}{{\rm lip}}
\newcommand{\nablaAM}{\nabla_{\!\scriptscriptstyle\rm AM}}
\newcommand{\Gr}{{\rm Gr}}
\newcommand{\Eucl}{{\rm Eucl}}
\newcommand{\Comp}{{\rm Comp}}
\newcommand{\Ch}{{\rm Ch}}
\newcommand{\ppi}{{\mbox{\boldmath\(\pi\)}}}
\newcommand{\sfE}{{\sf E}}
\newcommand{\EAM}{\sfE_{\scriptscriptstyle{\rm AM}}}
\newcommand{\limi}{\varliminf}
\newcommand{\lims}{\varlimsup}

\newcommand{\fr}{\penalty-20\null\hfill\(\blacksquare\)}

\newtheorem{theorem}{Theorem}[section]

\newtheorem{lemma}[theorem]{Lemma}
\newtheorem{proposition}[theorem]{Proposition}
\newtheorem{definition}[theorem]{Definition}
\newtheorem{example}[theorem]{Example}
\newtheorem{remark}[theorem]{Remark}

\title{A short proof of the infinitesimal Hilbertianity of the weighted Euclidean space}

\author{Simone Di Marino}
\address{Dipartimento di Matematica (DIMA), Via Dodecaneso 35, 16146 Genova, Universit\`a di Genova}
\email{simone.dimarino@unige.it}

\author{Danka Lu\v{c}i\'{c}}
\address{Department of Mathematics and Statistics,
P.O.\ Box 35 (MaD), FI-40014 University of Jyvaskyla}
\email{danka.d.lucic@jyu.fi}

\author{Enrico Pasqualetto}
\address{Department of Mathematics and Statistics,
P.O.\ Box 35 (MaD), FI-40014 University of Jyvaskyla}
\email{enrico.e.pasqualetto@jyu.fi}

\begin{document}
\date{\today} 
\keywords{Sobolev space, infinitesimal Hilbertianity, weighted Euclidean space, decomposability bundle, closability of the Sobolev norm}
\subjclass[2010]{53C23, 46E35, 26B05}
\begin{abstract}
We provide a quick proof of the following known result: the Sobolev space
associated with the Euclidean space, endowed with the Euclidean distance and
an arbitrary Radon measure, is Hilbert. Our new approach relies upon the
properties of the Alberti--Marchese decomposability bundle. As a consequence
of our arguments, we also prove that if the Sobolev norm is closable on
compactly-supported smooth functions, then the reference measure is
absolutely continuous with respect to the Lebesgue measure.
\end{abstract}
\maketitle
\section*{Introduction}
In recent years, the theory of weakly differentiable functions over an abstract
metric measure space \((\X,\sfd,\mu)\) has been extensively studied.
Starting from the seminal paper \cite{Cheeger00}, several (essentially equivalent) versions of Sobolev space \(W^{1,2}(\X,\sfd,\mu)\) have been proposed in \cite{Shanmugalingam00,AmbrosioGigliSavare11,DiM14a}.
The definition we shall adopt in this paper is the one via test plans and weak upper
gradients, which has been introduced by L.\ Ambrosio, N.\ Gigli and G.\ Savar\'{e}
in \cite{AmbrosioGigliSavare11}. In general, \(W^{1,2}(\X,\sfd,\mu)\) is a Banach space,
but it might be non-Hilbert: for instance, consider the Euclidean space
endowed with the \(\ell^\infty\)-norm and the Lebesgue measure. Those metric measure
spaces whose associated Sobolev space is Hilbert -- which are said to be
\emph{infinitesimally Hilbertian}, cf.\ \cite{Gigli12} -- play a very
important role. We refer to the introduction of \cite{LP20} for an account
of the main advantages and features of this class of spaces.
\medskip

The aim of this manuscript is to provide a quick proof of the following result
(cf.\ Theorem \ref{thm:Eucl_inf_Hilb}):
\begin{equation}\tag{\(\star\)}\label{eq:statement_main_result}
(\R^d,\sfd_\Eucl,\mu)\;\text{ is infinitesimally Hilbertian for any Radon measure }
\mu\geq 0\text{ on }\R^d,
\end{equation}
where \(\sfd_\Eucl(x,y)\coloneqq|x-y|\) stands for the Euclidean distance on \(\R^d\).
This fact has been originally proven in \cite{GP16-2}, but it can also be alternatively
considered as a special case of the main result in \cite{DMGSP18}. The approach we propose
here is more direct and is based upon the differentiability theorem \cite{AM16} for Lipschitz
functions in \(\R^d\) with respect to a given Radon measure, as we are going to describe.
\medskip

Let \(\mu\geq 0\) be any Radon measure on \(\R^d\). G.\ Alberti and
A.\ Marchese proved in \cite{AM16} that it is possible to select the maximal
measurable sub-bundle \(V(\mu,\cdot)\) of \(T\R^d\) -- called the \emph{decomposability bundle}
of \(\mu\) -- along which all Lipschitz functions are \(\mu\)-a.e.\ differentiable.
This way, any given Lipschitz function \(f\colon\R^d\to\R\) is naturally associated with
a gradient \(\nablaAM f\), which is an \(L^\infty\)-section of \(V(\mu,\cdot)\).
Being \(\nablaAM\) a linear operator, its induced Dirichlet energy functional
\(\EAM\) on \(L^2(\mu)\) is a quadratic form. Hence, the proof of
\eqref{eq:statement_main_result} presented here follows along these lines:
\begin{itemize}
\item[\(\rm a)\)] The maximality of \(V(\mu,\cdot)\) ensures that the curves selected
by a test plan \(\ppi\) on \((\R^d,\sfd_\Eucl,\mu)\) are `tangent' to \(V(\mu,\cdot)\),
namely, \(\dot\gamma_t\in V(\mu,\gamma_t)\) for
\((\ppi\otimes\mathcal L^1)\)-a.e.\ \((\gamma,t)\). See Lemma \ref{lem:pi_tangent_to_V}.
\item[\(\rm b)\)] Given any Lipschitz function \(f\colon\R^d\to\R\), we can deduce from
item a) that the modulus of the gradient \(\nablaAM f\) is a weak upper gradient of \(f\);
cf.\ Proposition \ref{prop:nablaAM_wug}.
\item[\(\rm c)\)] Since Lipschitz functions with compact support are dense in energy
in \(W^{1,2}(\R^d,\sfd_\Eucl,\mu)\) -- cf.\ Theorem \ref{thm:density_in_energy} below --
we conclude from b) that the Cheeger energy \(\sfE_\Ch\) is the lower semicontinuous
envelope of \(\EAM\). This grants that \(\sfE_\Ch\) is a quadratic form, thus accordingly
the space \((\R^d,\sfd_\Eucl,\mu)\) is infinitesimally Hilbertian.
See Theorem \ref{thm:Eucl_inf_Hilb} for the details.
\end{itemize}
Finally, by combining our techniques with a structural result
for Radon measures in the Euclidean space by De Philippis--Rindler
\cite{DPR}, we eventually prove (in Theorem \ref{thm:no_closable})
the following claim:
\[
\text{The Sobolev norm }\|\cdot\|_{W^{1,2}(\R^d,\sfd_{\rm Eucl},\mu)}
\text{ is closable on }C^\infty_c\text{-functions}\quad\Longrightarrow\quad\mu\ll\mathcal L^d.
\]
Cf.\ Definition \ref{def:closability} for the notion of closability we
are referring to. This result solves a conjecture that has been posed by
M.\ Fukushima (according to V.I.\ Bogachev \cite[Section 2.6]{Bogachev10}).
\medskip

{\bf Acknowledgements.} The second and third named authors acknowledge the support by
the Academy of Finland, projects 274372, 307333, 312488, and 314789.
\section{Preliminaries}
\subsection{Sobolev calculus on metric measure spaces}
By \emph{metric measure space} \((\X,\sfd,\mu)\) we mean a complete, separable
metric space \((\X,\sfd)\) together with a non-negative Radon measure \(\mu\neq 0\).
\medskip

We denote by \(\LIP(\X)\) the space of all real-valued Lipschitz functions on \(\X\),
whereas \(\LIP_c(\X)\) stands for the family of all elements of \(\LIP(\X)\) having
compact support. Given any \(f\in\LIP(\X)\), we shall denote by
\(\lip(f)\colon\X\to[0,+\infty)\) its \emph{local Lipschitz constant},
which is defined as
\[\lip(f)(x)\coloneqq\left\{\begin{array}{ll}
\lims_{y\to x}\big|f(x)-f(y)\big|/\sfd(x,y)\\
0
\end{array}\quad\begin{array}{ll}
\text{ if }x\in\X\text{ is an accumulation point,}\\
\text{ otherwise.}
\end{array}\right.\]
The metric space \((\X,\sfd)\) is said to be \emph{proper} provided its bounded, closed
subsets are compact.
\medskip

To introduce the notion of Sobolev space \(W^{1,2}(\X,\sfd,\mu)\) that has been proposed
in \cite{AmbrosioGigliSavare11}, we first need to recall some terminology.
The space \(C\big([0,1],\X\big)\) of all continuous curves in \(\X\) is a complete,
separable metric space if endowed with the sup-distance
\(\sfd_\infty(\gamma,\sigma)\coloneqq\max\big\{\sfd(\gamma_t,\sigma_t)\;\big|\;t\in[0,1]\big\}\).
We say that \(\gamma\in C\big([0,1],\X\big)\) is \emph{absolutely continuous}
provided there exists a function \(g\in L^1(0,1)\) such that
\(\sfd(\gamma_s,\gamma_t)\leq\int_s^t g(r)\,\d r\) holds for all \(s,t\in[0,1]\)
with \(s<t\). The \emph{metric speed} \(|\dot\gamma|\) of \(\gamma\), defined as
\(|\dot\gamma_t|\coloneqq\lim_{h\to 0}\sfd(\gamma_{t+h},\gamma_t)/|h|\)
for \(\mathcal L^1\)-a.e.\ \(t\in[0,1]\), is the minimal integrable function
(in the \(\mathcal L^1\)-a.e.\ sense)
that can be chosen as \(g\) in the previous inequality;
cf.\ \cite[Theorem 1.1.2]{AmbrosioGigliSavare08}.
A \emph{test plan} over \((\X,\sfd,\mu)\) is a Borel probability measure \(\ppi\)
on \(C\big([0,1],\X\big)\), concentrated on absolutely continuous curves,
such that the following properties are satisfied:
\begin{itemize}
\item \textsc{Bounded compression.} There exists \(\Comp(\ppi)>0\)
such that \((\e_t)_*\ppi\leq\Comp(\ppi)\,\mu\) holds for all \(t\in[0,1]\),
where \(\e_t\colon C\big([0,1],\X\big)\to\X\) stands for the evaluation map
\(\gamma\mapsto\e_t(\gamma)\coloneqq\gamma_t\).
\item \textsc{Finite kinetic energy.} It holds that
\(\int\!\!\int_0^1|\dot\gamma_t|^2\,\d t\,\d\ppi(\gamma)<+\infty\).
\end{itemize}
Let \(f\colon\X\to\R\) be a given Borel function. We say that \(G\in L^2(\mu)\)
is a \emph{weak upper gradient} of \(f\) provided for any test plan \(\ppi\)
on \((\X,\sfd,\mu)\) it holds that \(f\circ\gamma\in W^{1,1}(0,1)\) for
\(\ppi\)-a.e.\ \(\gamma\) and that
\[\big|(f\circ\gamma)'_t\big|\leq G(\gamma_t)\,|\dot\gamma_t|
\quad\text{ for }(\ppi\otimes\mathcal L^1)\text{-a.e.\ }(\gamma,t).\]
The minimal such function \(G\) (in the \(\mu\)-a.e.\ sense) is called the
\emph{minimal weak upper gradient} of \(f\) and is denoted by \(|Df|\in L^2(\mu)\).
\begin{definition}[Sobolev space \cite{AmbrosioGigliSavare11}]
The \emph{Sobolev space} \(W^{1,2}(\X,\sfd,\mu)\) is defined as the family of all
those functions \(f\in L^2(\mu)\) that admit a weak upper gradient \(G\in L^2(\mu)\).
We endow the vector space \(W^{1,2}(\X,\sfd,\mu)\) with the Sobolev norm
\(\|f\|_{W^{1,2}(\X,\sfd,\mu)}^2\coloneqq\|f\|_{L^2(\mu)}^2+\big\||Df|\big\|_{L^2(\mu)}^2\).
\end{definition}
The Sobolev space \(\big(W^{1,2}(\X,\sfd,\mu),\|\cdot\|_{W^{1,2}(\X,\sfd,\mu)}\big)\)
is a Banach space, but in general it is not a Hilbert space. This fact motivates the
following definition, which has been proposed by N.\ Gigli:
\begin{definition}[Infinitesimal Hilbertianity \cite{Gigli12}]
We say that a metric measure space \((\X,\sfd,\mu)\) is \emph{infinitesimally Hilbertian}
provided its associated Sobolev space \(W^{1,2}(\X,\sfd,\mu)\) is a Hilbert space.
\end{definition}
Let us define the \emph{Cheeger energy} functional
\(\sfE_\Ch\colon L^2(\mu)\to[0,+\infty]\) as
\begin{equation}\label{eq:def_E_Ch}
\sfE_{\rm Ch}(f):=\left\{\begin{array}{ll}
\frac{1}{2}\int|Df|^2\,\d\mu\\
+\infty
\end{array}\quad\begin{array}{ll}
\text{ if }f\in W^{1,2}(\X,\sfd,\mu),\\
\text{ otherwise.}
\end{array}\right.
\end{equation}
It holds that the metric measure space \((\X,\sfd,\mu)\) is infinitesimally
Hilbertian if and only if \(\sfE_\Ch\) satisfies the \emph{parallelogram rule}
when restricted to \(W^{1,2}(\X,\sfd,\mu)\), \emph{i.e.},
\begin{equation}\label{eq:parallelogram_id}
\sfE_\Ch(f+g)+\sfE_\Ch(f-g)=2\,\sfE_\Ch(f)+2\,\sfE_\Ch(g)
\quad\text{ for every }f,g\in W^{1,2}(\X,\sfd,\mu).
\end{equation}
Furthermore, we define the functional \(\sfE_\lip\colon L^2(\mu)\to[0,+\infty]\) as
\begin{equation}\label{eq:def_E_lip}
\sfE_\lip(f)\coloneqq\left\{\begin{array}{ll}
\frac{1}{2}\int\lip^2(f)\,\d\mu\\
+\infty
\end{array}\quad\begin{array}{ll}
\text{ if }f\in\LIP_c(\X),\\
\text{ otherwise.}
\end{array}\right.
\end{equation}
Given any \(f\in\LIP_c(\X)\), it holds that \(f\in W^{1,2}(\X,\sfd,\mu)\)
and \(|Df|\leq\lip(f)\) in the \(\mu\)-a.e.\ sense. This ensures that the
inequality \(\sfE_\Ch\leq\sfE_\lip\) is satisfied. Actually, \(\sfE_\Ch\)
is the \(L^2(\mu)\)-relaxation of \(\sfE_\lip\):
\begin{theorem}[Density in energy \cite{AmbrosioGigliSavare11-3}]
\label{thm:density_in_energy}
Let \((\X,\sfd,\mu)\) be a metric measure space, with \((\X,\sfd)\) proper.
Then \(\sfE_\Ch\) is the \emph{\(L^2(\mu)\)-lower semicontinuous envelope}
of \(\sfE_\lip\), \emph{i.e.}, it holds that
\[\sfE_\Ch(f)=\inf\limi_{n\to\infty}\sfE_\lip(f_n)\quad\text{ for every }f\in L^2(\mu),\]
where the infimum is taken among all sequences
\((f_n)_n\subseteq L^2(\mu)\) such that \(f_n\to f\) in \(L^2(\mu)\).
\end{theorem}
\subsection{Decomposability bundle}\label{ss:decomposability_bundle}
Let us denote by \(\Gr(\R^d)\) the set of all linear subspaces of \(\R^d\).
Given any \(V,W\in\Gr(\R^d)\), we define the distance \(\sfd_\Gr(V,W)\)
as the Hausdorff distance in \(\R^d\) between the closed unit ball of \(V\)
and that of \(W\). Hence, \(\big(\Gr(\R^d),\sfd_\Gr\big)\) is a compact metric space.
\begin{theorem}[Decomposability bundle \cite{AM16}]\label{thm:Alberti-Marchese}
Let \(\mu\geq 0\) be a given Radon measure on \(\R^d\). Then there exists
a \(\mu\)-a.e.\ unique Borel mapping \(V(\mu,\cdot)\colon\R^d\to\Gr(\R^d)\),
called the \emph{decomposability bundle} of \(\mu\), such that the following
properties hold:
\begin{itemize}
\item[\(\rm i)\)] Any function \(f\in\LIP(\R^d)\) is differentiable at
\(\mu\)-a.e.\ \(x\in\R^d\) with respect to \(V(\mu,x)\), \emph{i.e.}, there exists a Borel
map \(\nablaAM f\colon\R^d\to\R^d\) such that \(\nablaAM f(x)\in V(\mu,x)\)
for all \(x\in\R^d\) and
\begin{equation}\label{eq:formula_nablaAM}
\lim_{V(\mu,x)\ni v\to 0}\frac{f(x+v)-f(x)-\nablaAM f(x)\cdot v}{|v|}=0
\quad\text{ for }\mu\text{-a.e.\ }x\in\R^d.
\end{equation}
\item[\(\rm ii)\)] There exists a function \(f_0\in\LIP(\R^d)\) such that
for \(\mu\)-a.e.\ point \(x\in\R^d\) it holds that \(f_0\) is not differentiable
at \(x\) with respect to any direction \(v\in\R^d\setminus V(\mu,x)\).
\end{itemize}
\end{theorem}

We refer to \(\nablaAM f\) as the \emph{Alberti--Marchese gradient} of \(f\).
It readily follows from \eqref{eq:formula_nablaAM} that \(\nablaAM f\) is
uniquely determined (up to \(\mu\)-a.e.\ equality) and that for every
\(f,g\in\LIP(\R^d)\) it holds that
\begin{equation}\label{eq:nablaAM_linear}
\nablaAM(f\pm g)(x)=\nablaAM f(x)\pm\nablaAM g(x)
\quad\text{ for }\mu\text{-a.e.\ }x\in\R^d.
\end{equation}
\begin{remark}{\rm
Theorem \ref{thm:Alberti-Marchese} was actually proven under the
additional assumption of \(\mu\) being a finite measure. However,
the statement depends only on the null sets of \(\mu\), not on the
measure \(\mu\) itself. Therefore, in order to obtain Theorem
\ref{thm:Alberti-Marchese} as a consequence of the original result in \cite{AM16}, it is sufficient to
replace \(\mu\) with the following Borel probability measure on \(\R^d\):
\[
\tilde\mu\coloneqq\sum_{j=1}^\infty\frac{\mu|_{B_j(\bar x)}}
{2^j\mu\big(B_j(\bar x)\big)},\quad\text{ for some }\bar x\in{\rm spt}(\mu).
\]
Observe, indeed, that the measure \(\tilde\mu\) satisfies
\(\mu\ll\tilde\mu\ll\mu\).
\fr}\end{remark}
\begin{remark}{\rm
Given any function \(f\in\LIP(\R^d)\), it holds that
\begin{equation}\label{eq:nablaAM_leq_lip}
\big|\nablaAM f(x)\big|\leq\lip(f)(x)\quad\text{ for }\mu\text{-a.e.\ }x\in\R^d.
\end{equation}
Indeed, fix any point \(x\in\R^d\) such that \(f\) is differentiable at \(x\) with
respect to \(V(\mu,x)\). Then for all \(v\in V(\mu,x)\setminus\{0\}\) it holds that
\(\nablaAM f(x)\cdot v=|v|\,\lim_{h\searrow 0}\big(f(x+hv)-f(x)\big)/|hv|
\leq|v|\,\lip(f)(x)\) by \eqref{eq:formula_nablaAM}, thus accordingly
\(\big|\nablaAM f(x)\big|=\sup\big\{\nablaAM f(x)\cdot v\;\big|\;
v\in V(\mu,x),\,|v|\leq 1\big\}\leq\lip(f)(x)\).
\fr}\end{remark}
\section{Universal infinitesimal Hilbertianity of the Euclidean space}
The objective of this section is to show that the Euclidean space is
\emph{universally infinitesimally Hilbertian}, meaning that it is
infinitesimally Hilbertian when equipped with any Radon measure;
cf.\ Theorem \ref{thm:Eucl_inf_Hilb} below. The strategy of the proof we are going to
present here is based upon the structure of the decomposability bundle
described in Subsection \ref{ss:decomposability_bundle}.
\medskip

First of all, we prove that any given test plan over the weighted Euclidean space
is `tangent', in a suitable sense, to the Alberti--Marchese decomposability bundle:
\begin{lemma}\label{lem:pi_tangent_to_V}
Let \(\mu\geq 0\) be a given Radon measure on \(\R^d\).
Let \(\ppi\) be a test plan on \((\R^d,\sfd_\Eucl,\mu)\).
Then for \(\ppi\)-a.e.\ \(\gamma\) it holds that
\[\dot\gamma_t\in V(\mu,\gamma_t)\quad
\text{ for }\mathcal L^1\text{-a.e.\ }t\in[0,1].\]
\end{lemma}
\begin{proof}
Let \(f_0\) be an \(L\)-Lipschitz function as in ii) of Theorem \ref{thm:Alberti-Marchese}.
Set \(B\subseteq C\big([0,1],\R^d\big)\times[0,1]\) as
\[B\coloneqq\Big\{(\gamma,t)\;\Big|\;\gamma\text{ and }f_0\circ\gamma
\text{ are differentiable at }t,\text{ and }\dot\gamma_t\notin V(\mu,\gamma_t)\Big\}.\]
It can be easily shown that \(B\) is Borel measurable. We can assume that \(\gamma\) is
absolutely continuous (since by definition a test plan is concentrated on absolutely 
continuous curves); in particular, also \(f_0 \circ \gamma\) is absolutely continuous,
and thus both \(\gamma\) and \(f_0 \circ \gamma\) are differentiable
\(\mathcal{L}^1\)-almost everywhere. In particular, we are done if
we can prove that \((\ppi\otimes\mathcal L^1)(B)=0\).

Call
\(B_t\coloneqq\big\{\gamma\;\big|\;(\gamma,t)\in B\big\}\) for every \(t\in[0,1]\).
Moreover, \(G\) stands for the set of all \(x\in\R^d\) such that \(f_0\) is not
differentiable at \(x\) with respect to any direction \(v\in\R^d\setminus V(\mu,x)\).
Thus, \(\mu(\R^d\setminus G)=0\) by Theorem \ref{thm:Alberti-Marchese}.
We claim that the inclusion \(\e_t(B_t)\subseteq\R^d\setminus G\)
holds for every \(t\in[0,1]\). Indeed, for every \(\gamma\in B_t\) one has that
\[\begin{split}
\bigg|\frac{f_0(\gamma_t+h\dot\gamma_t)-f_0(\gamma_t)}{h}-(f_0\circ\gamma)'_t\bigg|
&\leq\bigg|\frac{f_0(\gamma_t+h\dot\gamma_t)-f_0(\gamma_{t+h})}{h}\bigg|
+\bigg|\frac{f_0(\gamma_{t+h})-f_0(\gamma_t)}{h}-(f_0\circ\gamma)'_t\bigg|\\
&\leq\,L\,\bigg|\frac{\gamma_{t+h}-\gamma_t}{h}-\dot\gamma_t\bigg|
+\bigg|\frac{f_0(\gamma_{t+h})-f_0(\gamma_t)}{h}-(f_0\circ\gamma)'_t\bigg|,
\end{split}\]
so by letting \(h\to 0\) we conclude that \(f_0\) is differentiable at \(\gamma_t\)
in the direction \(\dot\gamma_t\), \emph{i.e.}, \(\gamma_t\notin G\).
Therefore, we conclude that \(\ppi(B_t)\leq\ppi\big(\e_t^{-1}(\R^d\setminus G)\big)\leq
\Comp(\ppi)\,\mu(\R^d\setminus G)=0\) for all \(t\in[0,1]\). This grants that 
\((\ppi\otimes\mathcal L^1)(B)=0\) by Fubini theorem, whence the statement follows.
\end{proof}
As a consequence of Lemma \ref{lem:pi_tangent_to_V}, we can readily prove 
that the modulus of the Alberti--Marchese gradient of a given Lipschitz
function is a weak upper gradient of the function itself:
\begin{proposition}\label{prop:nablaAM_wug}
Let \(\mu\geq 0\) be a Radon measure on \(\R^d\).
Let \(f\in\LIP_c(\R^d)\) be given. Then the function
\(|\nablaAM f|\in L^2(\mu)\) is a weak upper gradient of \(f\).
\end{proposition}
\begin{proof}
Let \(\ppi\) be any test plan over \((\R^d,\sfd_\Eucl,\mu)\). We claim that
for \(\ppi\)-a.e.\ \(\gamma\) it holds
\begin{equation}\label{eq:nablaAM_wug_claim}
(f\circ\gamma)'_t=\nablaAM f(\gamma_t)\cdot\dot\gamma_t\quad
\text{ for }\mathcal L^1\text{-a.e.\ }t\in[0,1].
\end{equation}
Indeed, for \((\ppi\otimes\mathcal L^1)\)-a.e.\ \((\gamma,t)\) we
have that \(f\) is differentiable at \(\gamma_t\) with respect to \(V(\mu,\gamma_t)\)
and that \(\dot\gamma_t\in V(\mu,\gamma_t)\); this stems from item i) of Theorem
\ref{thm:Alberti-Marchese} and Lemma \ref{lem:pi_tangent_to_V}.
Hence, \eqref{eq:formula_nablaAM} yields
\[\nablaAM f(\gamma_t)\cdot\dot\gamma_t=
\lim_{h\searrow 0}\frac{f(\gamma_t+h\dot\gamma_t)-f(\gamma_t)}{h}=
\lim_{h\searrow 0}\frac{f(\gamma_{t+h})-f(\gamma_t)}{h}
=(f\circ\gamma)'_t,\]
which proves the claim \eqref{eq:nablaAM_wug_claim}.
In particular, for \(\ppi\)-a.e.\ curve \(\gamma\) it holds
\[\big|(f\circ\gamma)'_t\big|\leq\big|\nablaAM f(\gamma_t)\big|\,|\dot\gamma_t|
\quad\text{ for }\mathcal L^1\text{-a.e.\ }t\in[0,1].\]
Given that \(|\nablaAM f|\in L^2(\mu)\) by \eqref{eq:nablaAM_leq_lip},
we conclude that \(|Df|\leq|\nablaAM f|\) holds in the \(\mu\)-a.e.\ sense.
\end{proof}
We are now in a position to prove the universal infinitesimal Hilbertianity of
the Euclidean space, as an immediate consequence of Proposition \ref{prop:nablaAM_wug}
and of the linearity of \(\nablaAM\):
\begin{theorem}[Infinitesimal Hilbertianity of weighted \(\R^d\)]\label{thm:Eucl_inf_Hilb}
Let \(\mu\geq 0\) be a Radon measure on \(\R^d\).
Then the metric measure space \((\R^d,\sfd_{\rm Eucl},\mu)\) is infinitesimally Hilbertian.
\end{theorem}
\begin{proof}
First of all, let us define the \emph{Alberti--Marchese energy}
functional \(\EAM\colon L^2(\mu)\to[0,+\infty]\) as
\[\EAM(f)\coloneqq\left\{\begin{array}{ll}
\frac{1}{2}\int|\nablaAM f|^2\,\d\mu\\
+\infty
\end{array}\quad\begin{array}{ll}
\text{ if }f\in\LIP_c(\R^d),\\
\text{ otherwise.}
\end{array}\right.\]
Since \(|Df|\leq|\nablaAM f|\leq\lip(f)\) holds \(\mu\)-a.e.\ for
any \(f\in\LIP_c(\R^d)\) by Proposition \ref{prop:nablaAM_wug} and
\eqref{eq:nablaAM_leq_lip}, we have that \(\sfE_\Ch\leq\EAM\leq\sfE_\lip\),
where \(\sfE_\Ch\) and \(\sfE_\lip\) are defined as in \eqref{eq:def_E_Ch}
and \eqref{eq:def_E_lip}, respectively.
In view of Theorem \ref{thm:density_in_energy}, we deduce that \(\sfE_\Ch\)
is the \(L^2(\mu)\)-lower semicontinuous envelope of \(\EAM\). Thanks to
the identities in \eqref{eq:nablaAM_linear}, we also know that \(\EAM\) satisfies
the parallelogram rule when restricted to \(\LIP_c(\R^d)\), which means that
\begin{equation}\label{eq:parallelogram_id_AM}
\EAM(f+g)+\EAM(f-g)=2\,\EAM(f)+2\,\EAM(g)\quad\text{ for every }f,g\in\LIP_c(\R^d).
\end{equation}
Fix \(f,g\in W^{1,2}(\R^d,\sfd_\Eucl,\mu)\). Let us choose any
two sequences $(f_n)_n,(g_n)_n\subseteq\LIP_c(\R^d)$ such that
\begin{itemize}
\item \(f_n\to f\) and \(g_n\to g\) in \(L^2(\mu)\),
\item \(\EAM(f_n)\to\sfE_\Ch(f)\) and \(\EAM(g_n)\to\sfE_\Ch(g)\).
\end{itemize}
In particular, observe that \(f_n+g_n\to f+g\) and \(f_n-g_n\to f-g\) in \(L^2(\mu)\).
Therefore, it holds that
\[\begin{split}
\sfE_\Ch(f+g)+\sfE_\Ch(f-g)&\leq\limi_{n\to\infty}\big(\EAM(f_n+g_n)+\EAM(f_n-g_n)\big)
\overset{\eqref{eq:parallelogram_id_AM}}=2\lim_{n\to\infty}\big(\EAM(f_n)+\EAM(g_n)\big)\\
&=2\,\sfE_\Ch(f)+2\,\sfE_\Ch(g).
\end{split}\]
By replacing $f$ and $g$ with $f+g$ and $f-g$, respectively,
we conclude that the converse inequality is verified as well. Consequently,
the Cheeger energy \(\sfE_\Ch\) satisfies the parallelogram rule
\eqref{eq:parallelogram_id}, thus \(W^{1,2}(\R^d,\sfd_\Eucl,\mu)\) is a Hilbert space.
This completes the proof of the statement.
\end{proof}
\begin{remark}{\rm
As a byproduct of the proof of Theorem \ref{thm:Eucl_inf_Hilb}, we
see that for all $f\in W^{1,2}(\R^d,\sfd_{\rm Eucl},\mu)$
there exists a sequence $(f_n)_n\subseteq\LIP_c(\R^d)$ such that
$f_n\to f$ and $|\nablaAM f_n|\to|Df|$ in $L^2(\mu)$.
\fr}\end{remark}
\begin{example}\label{ex:example_Cantor}{\rm
Given an arbitrary Radon measure \(\mu\) on \(\R^d\), it might happen that
\[|Df|\neq|\nablaAM f|\quad\text{ for some }f\in\LIP_c(\R^d).\]
For instance, consider the measure \(\mu\coloneqq\mathcal L^1|_C\) on \(\R\),
where \(C\subseteq\R\) is any Cantor set of positive Lebesgue measure. Since the
support of \(\mu\) is totally disconnected, one has that every \(f\in L^2(\mu)\)
is a Sobolev function with \(|Df|=0\). However, it holds
\(V(\mu,x)=\R\) for \(\mathcal L^1\)-a.e.\ \(x\in C\) by Rademacher theorem,
whence for any \(f\in\LIP(\R)\) we have that \(\nablaAM f(x)=f'(x)\) for
\(\mathcal L^1\)-a.e.\ \(x\in C\).
\fr}\end{example}
\section{Closability of the Sobolev norm on smooth functions}
The aim of this conclusive section is to address a problem that has been
raised by M.\ Fukushima (as reported in \cite[Section 2.6]{Bogachev10}).
Namely, we provide a (negative) answer to the following question:
{\it Does there
exist a singular Radon measure \(\mu\) on \(\R^2\) for which the Sobolev
norm \(\|\cdot\|_{W^{1,2}(\R^2,\sfd_{\rm Eucl},\mu)}\) is closable on
compactly-supported smooth functions (in the sense of Definition
\ref{def:closability} below)?}

Actually, we are going to prove a stronger result:
{\it Given any Radon measure \(\mu\) on \(\R^d\)
that is not absolutely continuous with
respect to \(\mathscr L^d\), it holds that
\(\|\cdot\|_{W^{1,2}(\R^d,\sfd_{\rm Eucl},\mu)}\) is not closable on
compactly-supported smooth functions.} Cf.\ Theorem \ref{thm:no_closable}
below.
\medskip

Let \(f\in C^\infty_c(\R^d)\) be given. Then we denote by \(\nabla f\colon\R^d\to\R^d\)
its classical gradient.
Note that the identity \(|\nabla f|=\lip(f)\) holds.
Given a Radon measure \(\mu\) on \(\R^d\), it is immediate
to check that
\begin{equation}\label{eq:proj_grad}
\nablaAM f(x)=\pi_x\big(\nabla f(x)\big)\quad\text{ for }\mu\text{-a.e.\ }x\in\R^d,
\end{equation}
where \(\pi_x\colon\R^d\to V(\mu,x)\) stands for the orthogonal projection map.
We denote by \(L^2_\mu(\R^d,\R^d)\) the space of all (equivalence classes, up
to \(\mu\)-a.e.\ equality, of) Borel maps \(v\colon\R^d\to\R^d\) with \(|v|\in L^2(\mu)\).

It holds that \(L^2_\mu(\R^d,\R^d)\) is a Hilbert space if
endowed with the norm \(v\mapsto\big(\int|v|^2\,\d\mu\big)^{1/2}\).
\begin{definition}[Closability of the Sobolev norm on smooth functions]
\label{def:closability}
Let \(\mu\) be a Radon measure on \(\R^d\). Then the Sobolev norm
\(\|\cdot\|_{W^{1,2}(\R^d,\sfd_{\rm Eucl},\mu)}\) is \emph{closable on
compactly-supported smooth functions} provided the following property is verified:
if a sequence \((f_n)_n\subseteq C^\infty_c(\R^d)\) satisfies \(f_n\to 0\) in \(L^2(\mu)\)
and \(\nabla f_n\to v\) in \(L^2_\mu(\R^d,\R^d)\) for some element
\(v\in L^2_\mu(\R^d,\R^d)\), then it holds that \(v=0\).
\end{definition}
In order to provide some alternative characterisations of the above-defined
closability property, we need to recall the following improvement of
Theorem \ref{thm:density_in_energy} in the weighted Euclidean space case:
\begin{theorem}[Density in energy of smooth functions
\cite{GP16-2}]\label{thm:density_in_energy_smooth}
Let \(\mu\) be a Radon measure on \(\R^d\).
Then \(\sfE_\Ch\) is the \(L^2(\mu)\)-lower semicontinuous envelope
of the functional
\[
L^2(\mu)\ni f\longmapsto\left\{\begin{array}{ll}
\frac{1}{2}\int|\nabla f|^2\,\d\mu\\
+\infty
\end{array}\quad\begin{array}{ll}
\text{ if }f\in C^\infty_c(\R^d),\\
\text{ otherwise.}
\end{array}\right.
\]
\end{theorem}
\begin{lemma}\label{lem:equiv_closable}
Let \(\mu\) be a Radon measure on \(\R^d\).
Then the following conditions are equivalent:
\begin{itemize}
\item[\(\rm i)\)] The Sobolev norm \(\|\cdot\|_{W^{1,2}(\R^d,\sfd_{\rm Eucl},\mu)}\)
is closable on compactly-supported smooth functions.
\item[\(\rm ii)\)] The functional \(\sfE_\lip\) -- see \eqref{eq:def_E_lip} --
is \(L^2(\mu)\)-lower semicontinuous when restricted to \(C^\infty_c(\R^d)\).
\item[\(\rm iii)\)] The identity \(|Df|=|\nabla f|\) holds \(\mu\)-a.e.\ on \(\R^d\),
for every function \(f\in C^\infty_c(\R^d)\).
\end{itemize}
\end{lemma}
\begin{proof}\ \\
{\color{blue}\({\rm i)}\Longrightarrow{\rm ii)}\)} Fix any \(f\in C^\infty_c(\R^d)\)
and \((f_n)_n\subseteq C^\infty_c(\R^d)\) such that \(f_n\to f\) in \(L^2(\mu)\).
We claim that
\begin{equation}\label{eq:equiv_closable_claim}
\int|\nabla f|^2\,\d\mu\leq\limi_{n\to\infty}\int|\nabla f_n|^2\,\d\mu.
\end{equation}
Without loss of generality, we may assume the right-hand side
in \eqref{eq:equiv_closable_claim} is finite. Therefore, we can find a subsequence
\((f_{n_k})_k\) of \((f_n)_n\) and an element \(v\in L^2_\mu(\R^d,\R^d)\) such that
\(\lim_k\int|\nabla f_{n_k}|^2\,\d\mu=\limi_n\int|\nabla f_n|^2\,\d\mu\) and
\(\nabla f_{n_k}\rightharpoonup v\) in the weak topology of \(L^2_\mu(\R^d,\R^d)\).
By virtue of Banach--Saks theorem, we can additionally require that
\(\nabla\tilde f_k\to v\) in the strong topology of \(L^2_\mu(\R^d,\R^d)\),
where we set \(\tilde f_k\coloneqq\frac{1}{k}\sum_{i=1}^k f_{n_i}\in C^\infty_c(\R^d)\)
for all \(k\in\N\). Since \(\tilde f_k-f\to 0\) in \(L^2(\mu)\)
and \(\nabla(\tilde f_k-f)\to v-\nabla f\) in \(L^2_\mu(\R^d,\R^d)\),
we deduce from i) that \(v=\nabla f\). Consequently, we have
that \(\nabla f_n\rightharpoonup\nabla f\) in the weak topology
of \(L^2_\mu(\R^d,\R^d)\), thus proving \eqref{eq:equiv_closable_claim}
by semicontinuity of the norm. In other words, it holds that
\(\sfE_\lip(f)\leq\limi_n\sfE_\lip(f_n)\), which yields the
validity of item ii).\\
{\color{blue}\({\rm ii)}\Longrightarrow{\rm iii)}\)} Let
\(f\in C^\infty_c(\R^d)\) be given. Theorem
\ref{thm:density_in_energy_smooth} yields existence of a sequence
\((f_n)_n\subseteq C^\infty_c(\R^d)\) such that \(f_n\to f\)
and \(|\nabla f_n|\to|Df|\) in \(L^2(\mu)\). Therefore,
item ii) ensures that
\[
\frac{1}{2}\int|\nabla f|^2\,\d\mu
=\sfE_\lip(f)\leq\limi_{n\to\infty}\sfE_\lip(f_n)=\lim_{n\to\infty}
\frac{1}{2}\int|\nabla f_n|^2\,\d\mu=\frac{1}{2}\int|Df|^2\,\d\mu.
\]
Since \(|Df|\leq|\nabla f|\) holds \(\mu\)-a.e.\ on \(\R^d\),
we conclude that \(|Df|=|\nabla f|\), thus proving item iii).\\
{\color{blue}\({\rm iii)}\Longrightarrow{\rm i)}\)} We argue by
contradiction: suppose that there exists a sequence
\((f_n)_n\subseteq C^\infty_c(\R^d)\) such that \(f_n\to 0\)
in \(L^2(\mu)\) and \(\nabla f_n\to v\) in \(L^2_\mu(\R^d,\R^d)\)
for some \(v\in L^2_\mu(\R^d,\R^d)\setminus\{0\}\).
Fix any \(k\in\N\) such that \(\|\nabla f_k-v\|_{L^2_\mu(\R^d,\R^d)}
\leq\frac{1}{3}\|v\|_{L^2_\mu(\R^d,\R^d)}\). In particular,
\(\|\nabla f_k\|_{L^2_\mu(\R^d,\R^d)}
\geq\frac{2}{3}\|v\|_{L^2_\mu(\R^d,\R^d)}\). Let us define
\(g_n\coloneqq f_k-f_n\in C^\infty_c(\R^d)\) for every \(n\in\N\).
Since \(g_n\to f_k\) in \(L^2(\mu)\) and \(\nabla g_n\to\nabla f_k-v\)
in \(L^2_\mu(\R^d,\R^d)\) as \(n\to\infty\), we conclude that
\[
\|\nabla f_k\|_{L^2_\mu(\R^d,\R^d)}\geq\frac{2}{3}\,\|v\|_{L^2_\mu(\R^d,\R^d)}
>\frac{1}{3}\,\|v\|_{L^2_\mu(\R^d,\R^d)}\geq
\|\nabla f_k-v\|_{L^2_\mu(\R^d,\R^d)}=
\lim_{n\to\infty}\|\nabla g_n\|_{L^2_\mu(\R^d,\R^d)},
\]
whence \(\sfE_\lip(f_k)>\lim_n\sfE_\lip(g_n)\). This contradicts
the lower semicontinuity of \(\sfE_\lip\) on \(C^\infty_c(\R^d)\).
Consequently, item i) is proven.
\end{proof}
The last ingredient we need is the following result proven by
G.\ De Philippis and F.\ Rindler:
\begin{theorem}[Weak converse of Rademacher theorem
\cite{DPR}]\label{thm:DPR}
Let \(\mu\) be a Radon measure on \(\R^d\). Suppose all Lipschitz
functions \(f\colon\R^d\to\R\) are \(\mu\)-a.e.\ differentiable.
Then it holds that \(\mu\ll\mathcal L^d\).
\end{theorem}
We are finally in a position to prove the following statement
concerning closability:
\begin{theorem}[Failure of closability for singular measures]
\label{thm:no_closable}
Let \(\mu\geq 0\) be a given Radon measure on \(\R^d\).
Suppose that \(\mu\) is not absolutely continuous with respect
to the Lebesgue measure \(\mathcal L^d\). Then
the Sobolev norm \(\|\cdot\|_{W^{1,2}(\R^d,\sfd_{\rm Eucl},\mu)}\)
is not closable on compactly-supported smooth functions.
\end{theorem}
\begin{proof}
First of all, Theorem \ref{thm:DPR} grants the existence of a
Lipschitz function \(f\colon\R^d\to\R\) and a Borel set
\(P\subseteq\R^d\) such that \(\mu(P)>0\) and \(f\) is not
differentiable at any point of \(P\). Recalling Theorem
\ref{thm:Alberti-Marchese}, we then see that \(V(\mu,x)\neq\R^n\)
for \(\mu\)-a.e.\ \(x\in P\). Therefore, we can find a compact
set \(K\subseteq P\) and a vector \(v\in\R^d\) such that \(\mu(K)>0\)
and \(v\notin V(\mu,x)\) for \(\mu\)-a.e.\ \(x\in K\). Now pick
any \(g\in C^\infty_c(\R^d)\) such that \(\nabla g(x)=v\)
holds for all \(x\in K\). Then Proposition \ref{prop:nablaAM_wug}
and \eqref{eq:proj_grad} yield
\[
|D g|(x)\leq|\nablaAM\,g|(x)=\big|\pi_x\big(\nabla g(x)\big)\big|
=\big|\pi_x(v)\big|<|v|=|\nabla g|(x)
\quad\text{ for }\mu\text{-a.e.\ }x\in K,
\]
thus accordingly \(\|\cdot\|_{W^{1,2}(\R^d,\sfd_{\rm Eucl},\mu)}\)
is not closable on compactly-supported smooth functions by
Lemma \ref{lem:equiv_closable}. Hence, the statement is achieved.
\end{proof}
\begin{remark}{\rm
The converse of Theorem \ref{thm:no_closable} might fail.
For instance, the measure \(\mu\) described in Example
\ref{ex:example_Cantor} is absolutely continuous with respect
to \(\mathcal L^1\), but the Sobolev norm
\(\|\cdot\|_{W^{1,2}(\R,\sfd_{\rm Eucl},\mu)}\) is not closable
on compactly-supported smooth functions as a consequence of
Lemma \ref{lem:equiv_closable}.
\fr}\end{remark}
\def\cprime{$'$} \def\cprime{$'$}

\end{document}